\documentclass[10pt,  reqno]{amsart}

\usepackage{latexsym,verbatim}
\usepackage{amsthm}
\usepackage{amsmath}
\usepackage{amsfonts}
\usepackage{color}
\usepackage{enumerate}

\usepackage[colorinlistoftodos]{todonotes}
\usepackage{euscript}
\usepackage[all, cmtip]{xy}
\usepackage{latexsym,tabularx}
\usepackage{amscd}
\usepackage{graphics}
\usepackage{amssymb}
\usepackage{epsfig}
\usepackage{fancyhdr}

\newtheorem{proposition}{Proposition}[section]
\newtheorem{construction}{Construction}[section]

\newtheorem{lemma}{Lemma}[section]

\newtheorem{remark}{Remark}[section]

\begin{document}
\title{Semistable fibrations over an elliptic curve with only one singular fibre.}
\author{Abel Castorena }
\address{Centro de Ciencias Matem\'aticas. Universidad Nacional Aut\'onoma de M\'exico, Campus Morelia.
Apdo. Postal 61-3(Xangari). C.P. 58089. Morelia, Michoac\'an. M\'exico.\\}
\email{abel@matmor.unam.mx}

\author{Margarida Mendes Lopes}

\address{CAMGSD/Departamento de Matem\'atica, Instituto Superior T\'ecnico. Av. Rovisco Pais, 1049-001. Lisboa, Portugal.}
\email{ mmlopes@math.tecnico.ulisboa.pt}
\author{Gian Pietro Pirola}
\address{Dipartimento di Matematica,
Universit\`a di Pavia,
Via Ferrata, 1,
 27100 Pavia, Italy}
\email{gianpietro.pirola@unipv.it }
\thanks{
   }

\begin{abstract} In this work we describe a  construction  of semistable fibrations  over an elliptic curve with one unique singular fibre and we give effective examples using monodromy of curves. 
\end{abstract}

\maketitle

\section{Introduction}  Let $X$ be a  compact smooth complex algebraic surface.  A fibration of $X$ is a morphism with connected fibres $\phi:X\to B$, where $B$ is a smooth curve.   It is well known that if $g(B)\leq 1$ and the fibration is not isotrivial  (i.e. such that not all smooth fibres are isomorphic) then it has at least one singular fibre (cf.  \cite [Th\'eor\`eme  4]{Sz},  \cite{Pa}).  

In this short note we describe a  construction  of semistable fibrations  over an elliptic curve with one unique singular  fibre and we give effective examples (section 2).  We also establish some general properties of such fibrations (section 1). 
Note that the existence of  these semistable fibrations contrasts with Remark 3 in (\cite{Pa}) which states that any fibration over an elliptic curve has at least two singular fibres. This claim is said to be an immediate consequence of  arguments used in the  proof of Theorem 4 of (\cite{Pa}). However it is not clear how those arguments are used. On the other hand there exist fibrations of genus 3 over an elliptic curve with a unique singular reduced fibre (cf. \cite{Ish}).

Our construction starts from  constructing certain ramified (non Galois) covers  $C\to E$ of elliptic curves using monodromy  and showing that under suitable hypothesis the surface $C\times C$  has a fibration $\phi: C\times C\to E$ as required. 

\smallskip

\noindent {\bf Notations and conventions.}  We work over the complex numbers.  $\sim$ denotes numerical equivalence of divisors, whilst $\equiv$ denotes linear equivalence. The composition of permutations is done right to left. 

\section{Generalities for a fibration over an elliptic curve with one unique singular fibre}

Let  $E$ be an elliptic curve with origin $O_E\in E$, and let $K_E$ be the canonical bundle of $E$.
Let $X$ be a compact smooth complex surface and  $\phi: X\to E$ be a relatively minimal semistable fibration  with general fibre  $F$ a curve of genus $g\geq 2$.   
 
 Note that $X$ is a minimal surface, since any rational curve must be vertical.  By the same reason $X$ is not birational to a ruled surface.
 % Assume also that the general fibres of $\phi$ are not all isomorphic. Since $X$ is minimal and the fibration is  semistable this is the same as  saying that the fibration is not trivial (i.e.  $X$ is not isomorphic to the product $F\times E$).

%Denote by $F$ the general fibre of $\phi$. 
%, and let $F_0$ be the singular fibre of $\phi$. 
Let $\omega_{X|E}$ be the relative  dualizing sheaf, and let $K_X$ be the canonical divisor of $X$. Denote $\Delta(\phi)=\deg(\phi_\ast(\omega_{X|E}))$.  Since $K_E$  is trivial, the relative canonical divisor $K_{\phi}$ is linearly equivalent to $K_X$, that is, 

$$K_{\phi}\equiv K_X-\phi^\ast{K_E}\equiv K_X.$$ 

\noindent Since $E$ is of genus one, one has $\Delta(\phi)=\chi(\mathcal O_X)$ and $K^2_{X|E}=K^2_X$. Set $\chi:=\chi\mathcal (O_X)$ and let $c_2:=c_2(X)=\chi_{top}(X)$ be the topological Euler characteristic of $X$.  Remark that for a semistable fibration over an elliptic curve $c_2$ is exactly the number of nodes of the singular fibres. 

Assume also that the general fibres of $\phi$ are not all isomorphic (since $X$ is minimal and the fibration is  semistable this is the same as  saying that the fibration is not trivial).
Then, by Arakelov's theorem (see \cite{Be3}),  $K^2_X>0$.  Since $S$ is not a rational surface,   by the classification of surfaces  we conclude that    $S$ is of general type   and hence $\chi\geq 1$.  
%The slope $\lambda(\phi)$ of $\phi$ is $$\lambda(\phi)=\frac{K^2_{X|E}}{\Delta(\phi)}=\frac{K^2_X}{\chi}.$$ 

%Assume that $X$ is of general type. Then 
%one has  the Bogomolov-Miyaoka-Yau inequality $K^2_X\leq 9\chi$ and the Noether equality $K^2_X+c_2=12\chi.$ 
% Assume furthermore that $\phi$ is not trivial.
%  Since $g:=g_F\geq 2$, we are in the conditions of Theorem 2 of  [2, p. 459], so we have the inequality $\frac{4(g-1)}{g}\leq\lambda(\phi).$ 
  
%  Sumarizing,  $$\chi\frac{4(g-1)}{g}\leq K^2_X\leq 9\chi.$$

 \vskip1mm Define, as usual,  $q=h^1(\mathcal O_X)$ and $p_g=h^0(K_X)$ to  be respectively the irregularity and the geometric genus of $X$.

By  \cite{Be3}, $1\leq q< g+1$, where the  strict inequality comes from the fact that we are assuming that the fibration is not trivial.
We recall  the sharpening of Vojta's inequalities by Tan (see \cite{Ta}, Theorem 2 and Lemma 3.1),  that in our case (i.e. $\phi:X \to E$  relatively minimal, not trivial and semistable) says that when the number $s$ of singular fibres  is positive then
  
$$\chi< \frac{g}{2}s\hskip2mm\text{ and }\hskip2mm K_X^2< (2g-2) s.$$

Finally we recall Beauville's  formula for a reducible semistable fibre $F_0$ (see the proof of Lemme 1 of \cite{Benombre}): 
$$  n= g-g(N)+c-1$$   where $g$ is the genus of the general fibre,  $N$ is the normalisation of $F_0$, $g(N)=h^1(N, \mathcal O_N)$, $c$ is the number of components of $F_0$ and $n$ is the number of double points of $F_0$.
\vskip1mm

\begin{proposition} Let $X$ be a  compact smooth complex surface, $E$ an elliptic curve  and  $\phi:X \to E$  a semistable fibration with general fibre $F$  of genus $g$ and   with one unique singular fibre $F_0$. Then: 

\begin{enumerate}[i)]

\item $K_X^2< 2g-2;$ 
\item $\chi<  \frac{g}{2};$
\item $c_2<4g+2;$
\item if $\phi$ is stable, $c_2\leq 3g-3;$
\item  if $\phi$ is stable and $c_2= 3g-3,$ then $q=1$ and $\phi$ is the Albanese map;

 \item if $\phi$ is stable, $2<\frac{12\chi}{5}< g-1$. 
 \end{enumerate}
 
 In particular if $\phi$ is stable, $g\geq 4$.
\end{proposition} 

\begin{proof} If $\phi$ is semistable with one unique singular fibre, the inequalities $K^2_X< 2g-2$  and $\chi<  \frac{g}{2}$ follow from the above mentioned sharpening of Vojta's inequalities.  The inequality $c_2<4g+2$ is the proposition of \cite {Be4}.

Let $F_0$ be the unique singular fibre.  Note that $c_2$ is exactly the number $n$ of double points of $F_0$. If $\phi$ is stable, every component $\theta$ of $F_0$ satisfies $K_X\theta\geq 1$. Since $K_X F_0=2g-2$, the number $c$ of components of $F_0$ is at most  $2g-2$.  
Using the above mentioned Beauville's formula  we obtain
$$c_2=g-g(N)+c-1\leq g-g(N)+2g-2-1\leq 3g-3-g(N)$$ and so  $c_2\leq 3g-3$. 

If equality holds we conclude that $c=2g-2$ and that $g(N)=0$, i.e. every component of $F_0$ is a rational curve.  But then the Albanese map of $X$ contracts $F_0$ to a point and therefore also all fibres of $\phi$. Hence $\phi$ is the Albanese map of $X$ and $q=1$.

%If $\phi$ is stable the vanishing cycle on the smooth fibre $F$ near $F_0$ gives a set ${\gamma_i}_{i=1}^{c_2}$ of disjoint Jordan curves and pairwise non homotopically  equivalent,  it follows that  $c_2\leq 3g-3$ (Alternatively one can use the dual graph of a stable curve of genus $g$, ).

 Also, we obtain $12\chi= c_2+K^2< 5(g-1)$ and so $$ 2<\frac{12\chi}{5}<g-1,$$
where the first inequality comes from the fact that $X$  has $\chi\geq 1$.
\end{proof} 
\begin{remark}

{\rm  Note that if $\phi$ is stable with one unique singular fibre and $g=4$, then $\chi=1$, $K^2\leq 5$ and $c_2\leq 9$. Now $c_2=9$ means $K_X^2=3$.  From Debarre's inequality, \cite {De}, $K^2\geq 2p_g$ for irregular surfaces we conclude that $q=1$ and $\phi$ is the Albanese fibration. But (see \cite {CC}) for $K^2=3$, $p_g=q=1$  the genus of the general fibre of the Albanese fibration is $2$ or $3$. Therefore we have $4\leq K^2\leq 5$. 

We also remark that $\chi\geq 2$ implies $g\geq 6$. 

%In our minimal case (2) of Proposition 1.2 we have $g=9$.
}	
\end{remark}

\section{Covering maps and semistable fibrations.}

Let $E$ be an elliptic curve and let $O_E\in E$ be the origin. Let $C$ be a smooth curve of genus $g_C$ and let $f:C\to E$ be a covering map of degree $d$. 
We have the following 
\begin{construction}  Set $S=C\times C$, and for any map $f:C\to E$, let $\phi: S\to E$ $$ \phi(a,b)=f(a)-f(b)$$ be the difference map.
\end{construction}

Despite the fact that $S$ is a very simple surface and that the construction  is very elementary,  the map $\phi$ can be interesting. We start by giving a condition that assures  the connectedness of the fibres of $\phi.$

\begin{proposition} \label{cons} Assume that the map $f:C\to E$ is primitive of degree $d$ and $g_C>1$, (that is,  $f$ is not decomposable and not \'etale).  Then  the fibres of $\phi: S\to E$ are connected and the general fibre $F$ of the fibration $\phi$ has  genus \label{fibra} $g(F)=2(g_C-1)d+1$. \label{cons} 
\end{proposition} 

\begin{proof} We have to show that  $\phi$ has connected fibres. Suppose by contradiction that the general fibre $F$ of $\phi$ is not connected. By the Stein factorization theorem, there exists a smooth curve $Y$, a finite morphism $h:Y\to E$, $\deg h>1$, and a morphism $\widetilde\phi:S\to Y$ with connected fibres such that the following diagram is commutative:  
$$
\xymatrix{S\ar[r]^{\phi}\ar[dr]_{\widetilde\phi} & E\\ & Y\ar[u]_{h}.\\}
$$
Fix a point $q\in C$ such that $f(q)=O_E,$ consider the inclusion $i_q: C\to C\times C$, $i_q(p)=(p,q)$,  and let $C_1:=i_q(C)$. We have $\phi\circ i_q=f.$   Setting $\kappa=\widetilde\phi \circ i_q,$ we get a decomposition $f =\kappa\circ h$, but since $f$ is not decomposable, $\deg \kappa=1$ and $\deg h=d$.  It follows that $Y$ is isomorphic to $C$. Moreover, for  the general fibre of $\phi$  one has $$F=\sum_i^d F_i$$ and therefore numerically $F\sim d\widetilde F$, $\widetilde F=F_1$. 
Since $F\cdot C_1=d$ we get $\widetilde F\cdot C_1=1$. By symmetry we also have  $\widetilde F\cdot C_2=1$, where $C_2 =\{q\}\times C.$ 

 Since $K_S\sim (2g_C-2)C_1+(2g_C-2)C_2$, and $\widetilde F^2=0$ we obtain by the adjunction formula
$2g({\widetilde F})-2= 4g_C-4$, i.e. $g({\widetilde F})=2g_C-1$. But on the other hand  $\widetilde F\cdot C_j=1$ implies, by considering  one of the projections onto $C$ of $C\times C$ that $g({\widetilde F})=g_C$, a contradiction.

\end{proof}

We want to consider now the case when the map $f:C\to E$ has a unique branch point, which will assume to be  $O_E\in E$.

%To get some special cases we assume that the origin $O_E\in E$ is the unique branch point of $f$.

% The ramification class of $f$ is given by a conjugacy class 
%$\tau=(m_1)\cdots (m_r)$, where the $(m_i)$ denote disjoints cycles of length $m_i$ in the symmetric group $S_d$. Then the ramification divisor is $R=\sum_{j=1}^r(m_i-1)p_i$, and by Riemann-Hurwitz formula we have $2g_C-2=\sum_{j=1}^r(m_j-1)$, where $g_C$ is the genus of $C$.  
 
\begin{proposition} Assume that $f:C\to E$ is primitive and has  only $O_E$ as a critical value. Let $ F_0=\phi^{-1}(O_E)$ be the fibre over the origin. Then: 

\begin{enumerate} [i)] \item The map $\phi$  is smooth outside  $F_0$;
\item $F_0$ decomposes as $F_0=\Gamma+\triangle$, where $\triangle\subset S$ is the diagonal and $\Gamma=F_0-\triangle$;

 \item   $\triangle\cdot \Gamma=2g_C-2$.
\item If the ramification divisor of $f$ is reduced, then $F_0$ has only nodes as singularities and $\phi$ is a stable fibration with only one singular fibre. 
\end{enumerate} 
 
  \label{crit}\end{proposition}

\vskip2mm

\begin{proof} The only critical value of $f$ is the origin $O_E.$  Let $\eta\in H^0(E, K_E)$,  $\eta\neq 0$,  be a generator, then $\phi^\ast(\eta)=(\eta_1,\eta_2)=(f^\ast(\eta), -f^\ast(\eta))$. The critical locus $C(\phi)$ of $\phi$ is then defined by the vanishing of $\phi^\ast(\eta)$, that is,  
 $$C(\phi):=\{(p,q)\in S: p,q\in\text{zero locus of }f^\ast(\eta)\}.$$ In particular, $C(\phi) \subset F_0=\phi^{-1}(O_E)$ and therefore $O_E$ is the only critical value of $\phi.$ Since 
 $\triangle\subset F_0$,  $\triangle\cdot F_0=0$   and so 
$$\triangle\cdot \Gamma= \triangle\cdot  (F_0-\triangle)=-\triangle^2=2g_C-2.$$

Suppose that the ramification divisor of $f$ is reduced. Then in any ramification point $p_j\in C$ of $f$,  $f$ locally is the map $z\to z^2$. So, for $(p_1,p_2)\in C(\phi)$, we can find local coordinates $(z_i, U_{p_i})$,  $i=1,2$, on $C$ such that $z_i(p_i)=0$, and a local coordinate $(t,W)$ on $E$, $O_E\in W$ and $t(O_E)=0$, such that locally $f$ and $\phi$ are given by
  
%\begin{equation} \label{nodes}
$$f(z_i)=z_i^2,\hskip3mm \phi(z_1,z_2)=z_1^2-z_2^2.$$
 %\end{equation} The equation of 

\vskip2mm
\noindent Since $F_0\cap U_1\times U_2$ around $(p_1, p_2)$ is given by $z_1^2-z_2^2=0, $ this proves that the singularities of $F_0$ are nodes. Since $S$ does not contain any rational curve,  $\phi$ is a stable fibration.
\end{proof}

%\begin{remark} Our construction can be performed starting from two different maps $f: C\to E$ and $h:D\to E$ and taking the difference $F:C\times D\to E,\hskip1mm F(x,y):=f(x)-h(y)$. Most of the results (but for the structure of the singular fibre) hold in this more general situation.
%\end{remark}

\vskip4mm  

In view of Proposition \ref {crit},  to obtain with this construction explicit examples of semistable fibrations over an elliptic curve with one unique singular fibre, 
we need to find primitive covers $f:C\to E$ such that  $O_E$ is the only branch point and the ramification divisor $R$ of $f$ is  reduced.

%\noindent To construct such a coverings we fix a point $x_0\in E-\{O_E\}.$  The fundamental group $\pi_1(E-\{O_E\}, x_0):=\Pi$ it is a free group on two generators $x,y$. We remark that the local monodromy is given by the commutator $[x,y]$ of $x$ and $y.$  We want to construct a covering $f:C\to E$ such that $O_E\in E$ is the unique branch point, and the local monodromy of $f$ is given by two  permutations $\alpha $ and $\beta$ in $S_d$, such that the group $G$ generated by $\alpha$ and $\beta$ is transitive and 
%$[\alpha,\beta]$ is the product of disjoint $2n$ transpositions. The monodromy homomorphism of groups is the map
%$\rho: \Pi_1\to S_d$ given by $\rho(x)=\alpha, \rho(y)=\beta.$  As $G=Im\rho$ is  a transitive subgroup of $S_d$, then we have a regular covering map of degree $d$, $C_0\to E-\{O\}$ with $C_0$ connected, and the compactification of this covering gives a smooth curve $C$ and a covering map $f:C\to E$ of degree $d$ ramified only at $O_E\in E$ with local monodromy given by $[\alpha, \beta]$. By construction, the ramification divisor of $f$ is reduced. In any ramification point $p_j\in C$ of $f$, we have that $f$ locally is the map $z\to z^2.$   This happens if and only if the disjoint cycles of $[\alpha, \beta]$ are transpositions. By Riemann-Hurwitz we have $2g_C-2=2n$, so $g_C=n+1.$ 

\vskip 2mm

We recall (see \cite[Ch. III]{Mi})  that given  a smooth complex irreducible projective curve
$E$,  a finite subset $B$  of $E$ and any point $x_0\in E-B$, there is a one-to-one correspondence between the set of isomorphism classes of covering maps $f:C\to E$ of degree $d$ whose branch points lie in $B$ and the set of group homomorphisms $\rho:\pi_1(E-B, x_0)\to S_d$ with transitive image (up to conjugacy in $S_d$).  The covering map is primitive if $Im(\rho)$ is a primitive subgroup of $S_d$. 

 In the case at hand, i.e. where $E$ is an elliptic curve $E$ and $B=O_E$ is a single point,  the fundamental group $\Pi_1:=\pi_1(E-\{O_E\}, x_0)$ is a free group on two generators $x,y$.  The monodromy homomorphism of groups $\rho: \Pi_1\to S_d$ is given by $\rho(x)=\alpha, \rho(y)=\beta$ and the ramification type of  the covering $f:C\to E$ is encoded in the decomposition of the commutator  $[\alpha, \beta]$  in $S_d$ as   a product of disjoint cycles.  The ramification divisor $R$ of $f$ is reduced if and only if every cycle in such a  decomposition   has length $2$.
 
  So to give explicit examples it is enough to find transitive  and primitive subgroups $G\subset S_d$ generated by two permutations $\alpha,\beta,$ such that $[\alpha, \beta]$ is a product of disjoint transpositions. We give here some  examples. 
  
  \vskip3mm
  
\noindent{\bf{Examples :}} 
 
 \vskip1mm

1)  Consider the subgroup $G\subset S_4$ generated by the permutations $\alpha,\beta,$ where:

$$  \hskip0.8mm\alpha=(1,\hskip1mm 2, \hskip1mm 3), \hskip0.8mm \beta=(2,\hskip1mm 3,\hskip1mm 4). $$

Then   $G$ is obviously a transitive subgroup and $$ [\alpha,\beta]=\alpha\beta\alpha^{-1}\beta^{-1}=(1,\hskip1mm 4)(2,\hskip1mm 3)$$

 Since the order of a block divides $4$,  if $G$ were not primitive, we could have only two blocks preserved. On the order hand $\alpha$ has order $3$ and therefore  $\alpha$ cannot preserve or interchange blocks.

 %In alternative use Hokkaido Mathematical Journal Vol. 26 (1997) p. 435-438, lemma 3  saying this is alternating group 
%%begin{equation}
%d=4n+1,\ \ \alpha=(1\hskip0.5mm 2)(2\hskip0.5mm 3)\cdots(2n-1\hskip0.7mm 2n),\hskip0.8mm\beta=(1\hskip0.7mm 2n+1)(2\hskip0.7mm 2n+2)\cdots(2n\hskip0.7mm 4n+1); \label{exa2} \end{equation}

% $$[\alpha,\beta]=\alpha\beta\alpha^{-1}\beta^{-1}=(4n\hskip0.5 mm 4n-1)(4n-2\hskip0.5mm 4n-3)\cdots(2n+1\hskip0.5 mm 2n+2).$$

In this case,   $C$ is a curve of genus $2$  and the general fibre of $\phi$ has genus  $9$ (cf. Proposition \ref{cons}).

\vskip2mm

2) This example is due to Pietro Corvaja.
\vskip1mm
\noindent Consider the subgroup $G\subset S_8$ generated by the permutations $\alpha,\beta,$ where:
$$\alpha= (1,\hskip0.7mm 2,\hskip0.7mm 3,\hskip0.7mm 3, \hskip0.7mm 5,  \hskip0.7mm 6,  \hskip0.7mm 7),\hskip4mm\beta= (8,\hskip0.7mm 3,\hskip0.7mm 4,\hskip0.7mm 1, \hskip0.7mm 5,  \hskip0.7mm 6),$$ 

then we have

$$ [\alpha,\beta]= (1,\hskip0.5mm 5)(2,\hskip0.5mm 6)(3,\hskip0.7mm 4)(7,\hskip0.7mm 8).$$
\vskip1mm
Since $G$ is isomorphic to $PGL_2(F_7)$, $G$ is transitive and primitive.  In this case,   $C$ is a curve of genus $3$  and the general fibre of $\phi$ has genus  $33$ (cf. Proposition \ref{cons}).

\vskip2mm

3)  Consider the subgroup $G\subset S_d$,  $d=4n+1, n\geq 2$ generated by the permutations $\alpha,\beta,$ where:

$$ \alpha=(1,\hskip0.5mm 2)(3,\hskip0.5mm 4)\cdots(2n-1,\hskip0.7mm 2n),
\hskip0.8mm\beta=(1,\hskip0.7mm 2n+1)\hskip0.7mm (2,\hskip0.7mm 2n+2,\cdots 2n, \hskip0.7mm 4n,  \hskip0.7mm 4n+1)$$

Then 
$$[\alpha,\beta]=(1,\hskip0.5 mm 2)(3,\hskip0.5mm 4)\cdots(4n-1,\hskip0.5 mm 4n).$$

Again $G$ is clearly a transitive subgroup of $S_d$.  Let  $\gamma= (1,\hskip0.7mm 2n+1)$ and $\delta=(2,\hskip0.7mm 2n+2,\cdots 2n, \hskip0.7mm 4n,  \hskip0.7mm 4n+1)$.  Then $\beta=\gamma \delta=\delta\gamma$ and furthermore $\beta^{4n-1}=\gamma$. Hence $G=<\alpha, \beta>$ contains $\gamma$ and $\delta$.

  Let $B$ be a block preserved by $G$ containing $1$. $B$ has cardinality $k>1$, because $2n+1\in B$. Since $k$ must divide $4n+1$, $k$ is odd and so $B$ contains another element $x$, $x\neq 1$, $x\neq 2n+1$.
 Since $x$ is a fixed element for $\gamma$, $\gamma(B)=B$. Similarly, since  $1$ is a fixed element for $\delta$,  $\delta(B)=B$.  So also $\beta(B)=B$ and therefore $B=\{1,2,....,4n+1\}$, i.e. $G$ is primitive. 

Here for each $n$,  $g(C)=n+1$  and the general fibre of $\phi$ has genus  $2nd+1$ (cf. Proposition \ref{cons}).
\vskip3mm
{\bf{Remark. }}We finally remark that recently, motivated by the theory of the  mapping class group, Pietro Corvaja and Fabrizio Catanese constructed several new examples.

\vskip4mm

\noindent {\bf Acknowledgments. } We want to thank   Fabrizio Catanese and Pietro Corvaja for fruitful comments and for the example 2 which was suggested by Pietro Corvaja.  We also thank  Alexis G. Zamora for bringing to our attention the paper of A. N. Parsin cited in this work.      
\vskip1mm
\noindent The first author was supported by Research Grant PAPIIT IN100716 (UNAM, M\'exico). The second author is a member of  the CAMGSD of  T\'ecnico-Lisboa, University of Lisbon and  she was partially supported by  FCT (Portugal)  through projects PTDC/MAT- GEO/2823/2014 and  UID/MAT/ 04459/2013. The third author was supported by PRIN 2015 €œModuli spaces and Lie Theory, INdAM - GNSAGA and FAR 2016 (PV) €Variet\`a algebriche, calcolo algebrico,
grafi orientati e topologici€.

%%%%%%%%%%%%%%%%%%%%%%%%%%%%%%%%%%%%%%%%%%%%%%%%%%%%%%%%%%%%%%%%%


\begin{thebibliography}{}

%\bibitem{Be1} Arnaud Beauville,  {\em Algebraic surfaces.} London Mathematical Society Lectures Notes.

\bibitem{Benombre} A.~Beauville, {\em   Le nombre minimum de fibres singuli\'eres d'un courbe stable sur $\mathbb P^1$},  In: SŽ\'eminaire sur les pinceaux de courbes de genre au moins deux, (L. Szpiro, ed.), AstŽ\'erisque \textbf{86}  (1981), 97--108.
 \bibitem {Be3} A.~Beauville, {\em  L'in\'egalit\'e $p_g\geqslant2q-4$ pour les surfaces de type g\'en\'eral}, Appendix to [De],  Bull. Soc. Math. de France, \textbf{110} (1982), 343--346. 
 \bibitem {Be4} A.~Beauville, {\em The Szpiro inequality for higher genus fibrations}, Algebraic Geometry (a volume in memory of Paolo Francia), 61-63; de Gruyter (2002).
 \bibitem {C} F. Catanese, {\em Kodaira fibrations and beyond: methods for moduli theory}, Japanese Journal of Mathematics, September 2017, Volume 12, Issue 2, 
pp 91-174.
 \bibitem {CC} F. Catanese, C. Ciliberto, {\em Surfaces with $p_g=q=1$}, in ``Problems in the Theory
of Surfaces and their Classification'', Cortona, 1988, Symposia Math., \textbf{32} (1992), 49-79.
\bibitem {De} O.~Debarre, {\em In\'egalit\'es num\'eriques pour les surfaces de type g\'en\'eral}, Bull. Soc. Math. de France, \textbf{110} (1982), 319--342.
\bibitem {Ish} H. Ishida, {\em  Catanese-Ciliberto surfaces of fiber genus three with unique singular
fiber}, Tohoku Math. J. (2), 58 (2006), 33-69.
\bibitem{Mi} R. Miranda,  {\em Algebraic curves and Riemann surfaces}, Graduate Studies in Mathematics, Vol. {\bf 5}, American Mathematical Society.
 \bibitem{Pa} A. N. Parsin. {\em Algebraic curves over function fields. I},  Izv. Akad. Nauk. SSSR. Ser. Mat. {\bf 32} (1968), No. 5.
\bibitem{Sz} L.~Szpiro, {\em   Propri\'et\'es num\'eriques du faisceau dualizant relatif},  In: SŽ\'eminaire sur les pinceaux de courbes de genre au moins deux, (L. Szpiro, ed.), AstŽ\'erisque \textbf{86}  (1981), 44--78.
\bibitem{Ta} S-L Tan. { \em The minimal number of singular fibres of a semistable curve over $\mathbb P^1$},  Journal of Algebraic Geometry, {\bf 4}  (1995), 591-596.
%\bibitem{X} Xiao Gang, {\em Fibred Algebraic Surfaces with Low Slope.} Math. Ann. 276, 449-466 (1987).
\end{thebibliography}
\end{document}